\documentclass{daj}

\dajAUTHORdetails{%
  title = {Substitution tilings with transcendental inflation factor}, 
  author = {Dirk Frettl\"oh, Alexey Garber, and Neil Ma\~nibo},
  plaintextauthor = {Dirk Frettloh, Alexey Garber, Neil Manibo},
  plaintexttitle = {Substitution tilings with transcendental inflation factor}, 
  keywords = {Substitutions, Delone set, infinite alphabets, transcendental inflation factor, aperiodic tilings, quasicompact operator.},
}   

\dajEDITORdetails{%
   year={2024},
   number={11},
   received={18 July 2023},   
   published={8 November 2024},  
   doi={10.19086/da.125449},       
}   

\usepackage{amsmath,amsfonts,amssymb,amsthm}
\usepackage{enumitem}
\usepackage{dsfont}

\newcommand{\R}{\mathbb{R}}
\newcommand{\ab}{\mathbf{a}}
\newcommand{\A}{\mathcal{A}}
\newcommand{\N}{\mathbb{N}}
\newcommand{\Z}{\mathbb{Z}}
\newcommand{\dd}{\mathrm{d}}
\newcommand{\bbo}{\mathds{1}} 

\renewcommand{\geq}{\geqslant}
\renewcommand{\leq}{\leqslant}
\renewcommand{\phi}{\varphi}
\renewcommand{\rho}{\varrho}
\renewcommand{\epsilon}{\varepsilon}

\newtheorem{thm}{Theorem}[section]
\newtheorem{lem}[thm]{Lemma}
\newtheorem{cor}[thm]{Corollary}

\newtheorem{prop}[thm]{Proposition}
\theoremstyle{definition}
\newtheorem{definition}[thm]{Definition}
\newtheorem{exam}[thm]{Example}
\newtheorem{ques}[thm]{Question}
\theoremstyle{remark}
\newtheorem*{rem}{Remark}

\begin{document}

\begin{frontmatter}[classification=text]

\title{Substitution Tilings with Transcendental Inflation Factor}

\author[df]{Dirk Frettl\"oh}
\author[ag]{Alexey Garber\thanks{Partially supported by Alexander von Humboldt Foundation Fellowship}}
\author[nm]{Neil Ma\~nibo\thanks{Funded by the German Academic Exchange Service
(DAAD) via the Postdoctoral Researchers International Mobility Experience (PRIME) Fellowship programme}}

\begin{abstract}
For any $\lambda>2$, we construct a substitution on an infinite alphabet which gives rise to a substitution tiling with inflation factor $\lambda$. In particular, we obtain the first class of examples of substitutive systems with transcendental inflation factors that possess usual dynamical properties enjoyed by primitive substitutions on finite alphabets. We show that both the associated subshift and tiling dynamical systems are strictly ergodic, which is related to the quasicompactness of the underlying substitution operator. We also provide an explicit substitution with transcendental inflation factor $\lambda$. 
\end{abstract}
\end{frontmatter}

\section{Introduction} \label{sec:intro}

%
%


Tile substitutions are frequently used to generate interesting aperiodic tilings like
the Penrose tilings or the Ammann--Beenker tilings in the plane, or aperiodic tilings with icosahedral symmetry in three-dimensional space \cite{BG,Enc}. Moreover, spaces of (bi-)infinite words generated by symbolic substitutions are well studied with respect to their dynamical, topological, and geometrical properties \cite{BG,DOP,FS,Q}. Usually, these substitutions on tiles (resp.\ on letters) use finitely many different tile types (resp.\ finitely many letters). This simple fact, together with some common conditions like primitivity, already implies that the tilings (resp.\ infinite words) will necessarily have an inflation factor $\lambda$ that is an algebraic integer. 


This paper explores new phenomena that arise when one allows for infinitely many prototiles (resp.\ letters). In particular, in this case there exist tile substitutions with transcendental inflation factors. Here we restrict ourselves to tilings in one dimension, hence the prototiles are just intervals. Our main results are the following.

\begin{thm} \label{thm:allfactors}
For every real number $\lambda>2$, there exists a substitution $\varrho$ on an (infinite) compact alphabet $\mathcal{A}$ such that the inflation factor of the associated inflation rule is $\lambda$ and every element of the subshift $(X_{\varrho},\sigma)$ gives rise to a Delone set. 
\end{thm}

Since algebraic numbers are countable the following result is immediate.

\begin{cor} \label{cor:transcfactors}
There is a substitution $\varrho$ on an (infinite) compact alphabet $\mathcal A$ such that the inflation factor of the associated geometric substitution $\lambda$ is a transcendental number and every element of the subshift $(X_{\varrho},\sigma)$ gives rise to a Delone set. 
\end{cor}

We complement this existence result in Theorem~\ref{thm:TM-transc} by constructing an explicit example with  transcendental inflation factor. This relies on existing results on values of certain transcendental power series on algebraic parameters;
see \cite{Adam} for a survey on Mahler's method that guarantees the transcendence of these values.

The following result ensures that our substitutions are well behaved; that is, they have the usual nice properties one expects from the finite alphabet case. Here and throughout the sequel, $\mathbf{a}=(a_i)_{i\geqslant 0}$ is a sequence of nonnegative integers satisfying certain properties, which we use to construct the substitution $\varrho$; see Section~\ref{sec:comp} for details on the construction.

\begin{thm}\label{thm:dyn}
For the substitution $\varrho:=\varrho_\mathbf a\colon \mathcal A\to \mathcal{A}^{+} $,
\begin{enumerate}
\item[\textnormal{(i)}] The subshift $(X_{\varrho},\sigma)$ is minimal and uniquely ergodic.
\item[\textnormal{(ii)}]  There is a unique (up to scaling) strictly positive, continuous, natural tile length function which turns $\varrho$ into a geometric substitution in $\R$. 
\item[\textnormal{(iii)}] The associated tiling dynamical system $(\Omega_{\varrho},\mathbb{R})$ is minimal and uniquely ergodic.
\end{enumerate}
\end{thm}

The term \emph{natural tile length} will become more clear later. But one can get an idea from Figure \ref{fig:inflation}: The natural tile lengths allow all inflated tiles (scaled by the inflation factor $\lambda$) to be dissected into copies of the original prototiles $a$ and $b$. Clearly, in this example,  only this particular length ratio $(1+\sqrt{5})/2$ does the job; whereas for instance lengths of 1 and 1, or of 1 and 2, will not. 

Substitutions on infinite alphabets or infinitely many prototiles have been studied before; see \cite{Fer,FS}. However, most known examples are either of constant length (which means all prototiles will just have the same length) or already have inherent pre-defined tile lengths. In \cite{MRW2}, it was asked whether, for substitutions over compact alphabets, there are conditions which guarantee the existence of a natural length function, from which one can recover an associated geometric inflation system. In this context, the following question was posed in \cite{MRW}:

\begin{ques}
Does there exist a primitive substitution on a compact alphabet $\A$ which admits a unique strictly positive (continuous) length function $\ell:\mathcal A\longrightarrow \R_+$ and a transcendental inflation factor $\lambda$?
\end{ques}


We answer this question in the affirmative in this work by Theorem \ref{thm:allfactors}, resp.\ Corollary~\ref{cor:transcfactors}.

 A point set $\varLambda\subseteq \R^d$ is called a {\it Delone set} if it is uniformly discrete and relatively dense; see Section~\ref{sec:infl} for a complete definition. In the specific case $d=1$ (which we mostly discuss in this paper), a point set $\varLambda\subset \R$ is  Delone if and only if the distance between two consecutive points of $\varLambda$ is uniformly bounded from above and from below by two positive numbers.

Delone sets arise in many topics related to mathematical crystallography, to distribution of points in Euclidean spaces, and in applications as well-spaced point sets. We refer to \cite{Cla,10R,FGS,Lag,LW} and references therein for detailed discussions on these sets and related properties. We want to emphasize their importance as point sets that can be used to model atomic structure of solids or ordered matter.

One natural connection between Delone sets and substitutions on finite alphabets is the following. Under certain conditions on the substitution, one can associate to it a strictly positive, bounded, length function $\ell$, which transforms every letter $a$ into a segment of length $\ell(a)$ and preserves the order of letters. 
Additionally, the length function $\ell$ respects the substitution in the sense that every inflated segment can be sliced into segments corresponding to the letters in the alphabet, thus turning the substitution into an inflation rule; see Figure~\ref{fig:inflation}.
This also turns every bi-infinite word into a tiling $\mathcal{T}$ of $\mathbb{R}$ consisting of finitely many distinct prototiles,   We call such a length function a \emph{natural length function}. One can then derive a  Delone set $\varLambda\subset \R$ by identifying a tile with the location of its left endpoint; see Figure~\ref{fig:tiling-Delone} for an example.  

\begin{figure}[!h]
\begin{center}
\includegraphics[scale=0.7]{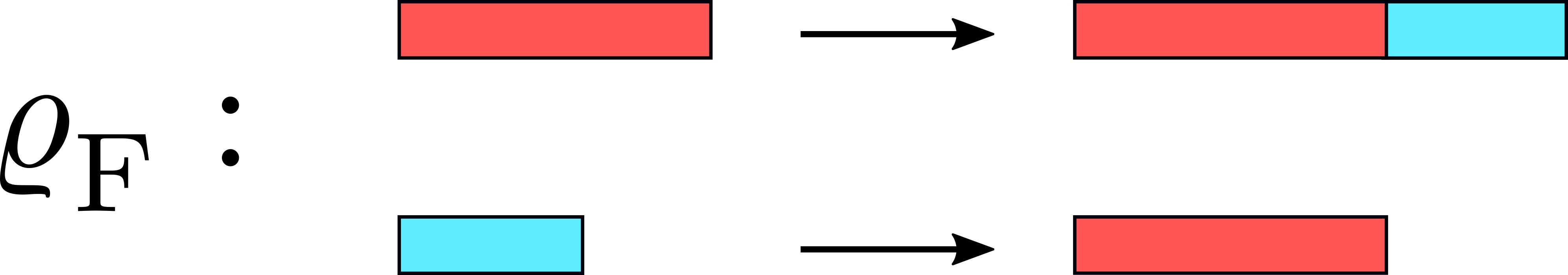} 
\end{center}
\caption{The geometric inflation rule associated to the Fibonacci substitution $\varrho^{ }_{\text{F}}$ given by $a\mapsto ab, b\mapsto a$. Here, the inflation factor is $\lambda=(1+\sqrt{5})/2$, and the natural length function $\ell$ assigns intervals of length $\lambda$ and $1$ to the letters $a$ and $b$, respectively. }\label{fig:inflation}

\end{figure}

\begin{figure}[!h]
\begin{center}
\[
\ldots abaababa|abaababa\ldots\quad \,
\]
\vspace{0.1cm}

\includegraphics[scale=0.75]{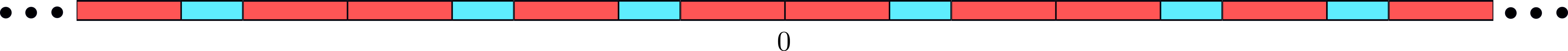} 
\vspace{0.5cm}

\includegraphics[scale=0.75]{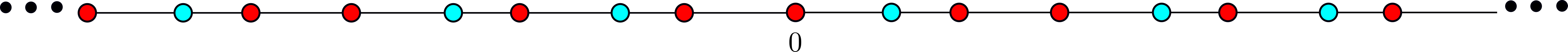}
\end{center} 
\caption{A bi-infinite word fixed under (the square of) the Fibonacci substitution, the corresponding tiling $\mathcal{T}$ of $\mathbb{R}$, and the derived Delone set $\varLambda$. }\label{fig:tiling-Delone}
\end{figure}

With a proper choice of bi-infinite word and origin, the resulting Delone set $\varLambda\subset \R$ exhibits inflation symmetry, that is, $\lambda \varLambda\subset \varLambda$, where $\lambda>1$ is (a positive integer power of) the inflation factor of the initial substitution, and hence algebraic.

Regarding the opposite direction: Let $\varLambda \subset \R^d$ be a Delone set, and let $\lambda \varLambda\subset \varLambda$ for some $\lambda >1$. Lagarias showed in \cite{Lag}: if $\varLambda$ is of finite type (that is, $\varLambda-\varLambda$ is a discrete closed subset of $\R^d$), then $\lambda$ must be an algebraic number.

It is not complicated to construct a Delone set $\varLambda\subset \R$ with a similar inclusion $\lambda\varLambda\subset \varLambda$ for a non-algebraic $\lambda$, but the additional structure of the family of Delone sets with the same ``local structures'' is not guaranteed, as they no longer come from substitutions on finite alphabets, and hence need not possess analogous dynamical properties. 


If the inflation factor $\lambda$ is not only an algebraic integer, but an (irrational) Pisot number
(that is, all algebraic conjugates of $\lambda$ are less than one in modulus), one gets a canonical description of the Delone set $\varLambda$ obtained from the substitution tiling by a cut-and-project scheme: the Galois conjugates of $\lambda$ yield the Minkowski embedding of $\varLambda$ as well as the star-map, which, in turn, yields the window for $\varLambda$, hence all components of the cut-and-project scheme for $\varLambda$. The (infinitely many) Delone sets produced by this cut-and-project scheme are called \emph{model sets}, or \emph{cut-and-project sets}. For details and terminology on cut-and-project schemes see \cite{BG}.  

The Pisot property is also central in the spectral analysis of the associated dynamical systems. For example, Solomyak showed in 
\cite{Sol} that for one-dimensional self-affine tilings with FLC (finite local complexity) with inflation factor $\lambda$, the corresponding dynamical system admits non-trivial eigenfunctions if and only if $\lambda$ is a Pisot number; see also \cite{BT,GK}. This uses the fact that $\alpha\in\mathbb{R}$ corresponds (in a certain sense) to an eigenvalue if and only if $||\alpha\lambda^{n}||\to 0$ as $n\to\infty$, where $||\cdot||$ denotes the distance to the closest integer. When $\lambda$
is algebraic, this only holds when $\lambda$ is a Pisot number. 
The still open Pisot--Vijayaraghavan problem asks whether there exist a transcendental $\lambda$ and a non-zero real $\alpha$ for which such a Diophantine convergence result holds; see \cite{B}. The work in this paper is interesting in view of this problem, and showing that our Delone sets do (or do not) admit non-trivial eigenfunctions might shed light to more restrictions such a transcendental $\lambda$ must satisfy.

%

The paper is organized as follows. In Section \ref{sec:alph}, we recall notions and properties of subshifts and substitutions on compact alphabets. We also recall the properties of the substitution operator (which is the infinite-alphabet analogue of the substitution matrix) and its dynamical consequences for the subshift. 


In Section \ref{sec:comp}, given a bounded sequence $\ab=(a_i)_i$ of nonnegative integers with an additional mild condition, we construct a map $\varrho_{\ab}$ from the set of letters $\N_0=\{[0],[1],[2],[3],\ldots\}$ to the set of finite words over $\N_0$. Here, $\N_0$ represents a (sub)set of letters of our alphabet that corresponds to nonnegative integers. We use the notation $[i]$ for the $i$th letter to distinguish it from the nonnegative integer $i$. We then define an appropriate embedding $\iota_{\ab}(\N_0)$ of $\N_0$ into a shift space $\mathcal{S}$ over a finite alphabet and consider its  compactification $\mathcal{A}:=\overline{\iota_{\ab}(\N_0)}$ (with respect to the topology on $\mathcal{S}$) as our alphabet. We extend the map $\varrho_{\ab}$ to a continuous map on $\mathcal{A}$, which will be our substitution. This allows us to use results of Ma\~{n}ibo, Rust, and Walton \cite{MRW2} on the associated dynamical system and enables us to construct a geometric substitution in $\R$. We prove Theorem \ref{thm:dyn} in this section. 


In Section \ref{sec:infl}, we provide a closed form for the inflation factor $\lambda$ of the substitution in terms of the sequence $\ab=(a_i)_{i\geqslant 0}$ and show that, for every $\lambda>2$, there is a suitable sequence $(a_i)_{i\geqslant 0}$ that gives rise to the inflation factor $\lambda$. This allows us to establish Theorem \ref{thm:allfactors}. In Section \ref{sec:periodic}, we show that if $(a_i)_{i\geqslant 0}$ is a periodic sequence, then the resulting inflation factor is algebraic. We also provide a counterexample for the converse. Moreover, we provide a concrete example of a substitution with transcendental inflation factor in Section~\ref{sec:trans}. We end with some open questions and concluding remarks.

\section{Substitutions on compact alphabets} \label{sec:alph}

In this section, we adapt the notations and definitions in \cite{MRW2}. 
\subsection{Subshifts and substitutions}


Let $\mathcal{A}$ be a compact Hausdorff space. The set $\A$ will be our \emph{alphabet}. Its elements are called \emph{letters}.  Let $\mathcal{A}^n$ be the collection of all words of length $n$ over $\mathcal{A}$, i.e., $w=w_1w_2\cdots w_n$, with $w_i\in\mathcal{A}$. 
We say that a 
word $u$ is a subword of $w$ (which we denote by $u\triangleleft w$) if there exist $1\leqslant i \leqslant j\leqslant n$ such that $u=w_iw_{i+1}\cdots w_{j}= w_{[i,j]}$. Note that the relation of ``being a subword'' can be easily generalized when $w$ is one-sided infinite or bi-infinite.
The set $\mathcal{A}^{+}=\bigcup_{n\geqslant 1} \mathcal{A}^n$ 
of all (non-empty) finite words over $\mathcal{A}$ is also a topological space (since the sets $\mathcal{A}^n$ are naturally endowed with the product topology), where the topology on $\mathcal{A}^{+}$ is that of a disjoint union. Let $\mathcal{A}^{\ast}=\mathcal{A}^{+}\cup \left\{\varepsilon\right\}$, where $\varepsilon$ is the empty word. 
$\A^*$ is a free monoid with concatenation as the binary operation. 

\begin{definition}
A {\it substitution} $\varrho$ is a continuous map from $\mathcal{A}$ to $\mathcal{A}^{+}$. 
\end{definition}

For every letter $a\in \mathcal A$, we call $\varrho^{k}(a)$ the \emph{level-$k$ supertile} of type $a$. 
We define the \emph{language} $\mathcal{L}(\varrho)$ of the substitution to be the set 

\[
\mathcal{L}(\mathcal{\varrho})=\overline{\left\{
u\in \mathcal{A}^{\ast}\mid u \triangleleft \varrho^{k}(a), k\in\mathbb{N}_0,a\in\mathcal{A}\right\}}. 
\]
So, the language $\mathcal{L}(\varrho)$ is the closure of the family of all subwords of supertiles.

This allows one to define a subshift associated to $\varrho$. Let $\mathcal{A}^{\mathbb{Z}}$ be the set of all bi-infinite sequences over $\mathcal{A}$, also called the \emph{full shift} over $\mathcal{A}$, which is compact with respect to the Tychonoff topology. Note that the left shift map $\sigma\colon \mathcal{A}^{\mathbb{Z}}\to \mathcal{A}^{\mathbb{Z}}$, defined pointwise via $\sigma(x)_n=x_{n+1}$, is a homeomorphism. 
A nonempty subset $X\subseteq \mathcal{A}^{\mathbb{Z}}$ is called a \emph{subshift} if it is closed and shift-invariant, i.e., $\sigma(X)\subseteq X$. 
We now define the subshift associated to $\varrho$ via its corresponding language, i.e., $X_{\varrho}:=\left\{x\in\mathcal{A}^{\mathbb{Z}}\mid u\in\mathcal{L}(\varrho), \text{ for all }u\triangleleft x\right\}$; see \cite{MRW2}. It can be shown that $\sigma(X_{\varrho})=X_{\varrho}$. 
A substitution $\varrho$ is said to be \emph{recognizable} if for every power $k\geqslant 1$, every element $x$ of the subshift admits a unique decomposition into level-$k$ supertiles of $\varrho$.

\begin{definition}
A substitution is called \emph{primitive} if for every non-empty open set $U\subset \mathcal{A}$, there exists $k(U)\in \mathbb{N}$ such that $\varrho^{k(U)}(a)$ contains an element in $U$ for all $a\in\mathcal{A}$. 
\end{definition}

Primitivity has several immediate consequences for the corresponding subshift, some of which we mention in the following result, see \cite[Thm. 3.30]{MRW2}. 

\begin{prop}
Let $\varrho$ be a primitive substitution over a compact Hausdorff alphabet. Then 
\begin{enumerate}
\item[\textnormal{(i)}] The subshift $X_{\varrho}$ is non-empty.
\item[\textnormal{(ii)}] The topological dynamical system $(X_{\varrho},\sigma)$ is minimal, i.e., every element $x\in X_{\varrho}$ has a dense orbit. 
\end{enumerate}
\end{prop}

\subsection{The substitution operator}

Let $\varrho\colon \mathcal{A}\to \mathcal{A}^{+}$ be a substitution over a compact Hausdorff alphabet. One can associate to it a substitution operator $M:=M_{\varrho}$ on the space $C(\mathcal{A})$ of real-valued continuous functions over $\mathcal{A}$
via
\[
Mf(a)=\sum_{b\,\triangleleft\,\varrho(a)} f(b),
\]
where the summation goes over all letters $b$ and $\varrho(a)$ is seen as a multiset, i.e. the summation takes into account multiplicities of letters in $\varrho(a)$ as well. 
This is the compact alphabet analogue of the transpose of the substitution matrix in the finite alphabet setting. This is a bounded operator because the length of $\varrho(a)$ is bounded. 
Let $\mathcal{K}$ be the positive cone of $C(\mathcal{A})$, i.e., all functions with $f(a)\geqslant 0$ for all $a\in\mathcal{A}$. One has $M\mathcal {K}\subseteq \mathcal{K}$, that is, $M$ is a positive operator on $C(\mathcal{A})$. 

The generalization of the Perron--Frobenius eigenvalue in this setting is the spectral radius $r(M)$ of $M$. Since $M$ is positive, $r(M)$ is always in the spectrum of $M$ but it is not guaranteed that it is an eigenvalue.
We refer the reader to \cite{Sch} for a background on positive operators on Banach lattices. 
 For the substitution operator, one has the following bounds for $r(M)$. 

\begin{lem}[\cite{MRW2}, Lemma 4.23]\label{lem:spec-rad}
For all $k\in\mathbb{N}$ one has
\[
\min_{a\in\mathcal{A}} |\varrho^{k}(a)|\leqslant r(M)^{k}\leqslant \max_{a\in\mathcal{A}} |\varrho^{k}(a)|. 
\]
\end{lem}

An operator $T$ with $r(T)=1$ over a Banach space  is called \emph{quasicompact} if there exists $k\in\mathbb{N}$ and a compact operator $K$ such that $\|T^k-K\|<1$, where $\| \cdot \|$ is the operator norm induced by the sup-norm on $C(\mathcal{A})$.

\begin{prop}[\cite{MRW2}, Theorem 6.8(1)] \label{prop:quasicompact}
Let $\varrho$ be a primitive substitution on a compact Hausdorff alphabet $\mathcal{A}$ for which there exists $k\in\mathbb{N}$ and a finite subset $F\subset \mathcal{A}$ of isolated points such that
\[
\left|\left\{b\in\varrho^{k}(a)\mid b\notin F\right\}\right|<r(M)^k
\]
for all $a\in\mathcal{A}$. Then the normalized substitution operator $T:=M/r(M)$ is quasicompact. 
\end{prop}

\begin{thm}[\cite{MRW2}, Theorem 6.7(3,4,5)]\label{thm:unique-ergod}
Let $\varrho$ be a primitive substitution on a compact Hausdorff alphabet. Suppose further that $T=M/r(M)$ is quasicompact. Then the following hold
\begin{enumerate}[label={{\rm (\roman*)}}]
\item $\varrho$ admits a unique natural tile length function which is strictly positive.
\item The corresponding tiling dynamical system $(\Omega_{\varrho},\mathbb{R})$ is uniquely ergodic.
\item The corresponding symbolic dynamical system $(X_{\varrho},\sigma)$ is  uniquely ergodic. 
\end{enumerate}
\end{thm}

\section{The substitution construction and its dynamical properties} 
\label{sec:comp}

%
%

\subsection{Pre-substitution}

Let $\mathbf a=(a_i)_{i\geqslant 0}=a_0,a_1,a_2,\ldots$ be a sequence of nonnegative integers that satisfies the following additional properties.

\begin{enumerate}[label={(A\arabic*)}]
\item \label{def:a1} The sequence $\mathbf a$ is bounded, where we set $N:=\max a_i$.
\item \label{def:a2} $a_0\neq 0$.
\item \label{def:a3} The runs of zeros in $\mathbf a$ are bounded. That is, there exists an integer $C>0$ such that if $a_i=a_{i+1}=\ldots=a_{i+k}=0$, then $k<C$.
\end{enumerate}


\begin{definition} \label{def:substitution}

Let $\mathbf a\in \left\{0,1,\ldots,N\right\}^{\mathbb{N}_0}$ be given and let $[i]\in\mathbb{N}_0$ be a letter in our pre-alphabet. We define the {\it pre-substitution} $\varrho^{ }_\mathbf a\colon \N_0\to \N_0^{+}$ via
$$
\begin{array}{l}
\varrho^{ }_\mathbf a([0])=[0]^{a_0}[1], \text{ and}\\
\varrho^{ }_\mathbf a([i])=[0]^{a_i}[i-1][i+1] \text{ for }i>0.
\end{array}
$$
\end{definition}

Here $[0]^{a_i}$ denotes $a_i$ zeros in a row.
We call the rule $\varrho^{ }_{\mathbf{a}}$ a pre-substitution because we will later extend it to a full substitution on a compact alphabet. This is possible if one embeds $\mathbb{N}_0$ properly. 
To this pre-substitution, one can associate the corresponding infinite matrix 
$$\mathbf A=\left( 
\begin{array}{cccccc}
a_0& a_1+1& a_2 & a_3 & a_4 & \dots\\
1 & 0 & 1 & 0 & 0 & \ldots\\
0 & 1 & 0 & 1 & 0 & \ldots\\
0 & 0 & 1 & 0 & 1 & \ldots\\
0 & 0 & 0 & 1 & 0 & \ldots\\
\vdots & \vdots & \vdots & \vdots & \vdots & \ddots
\end{array}
\right),$$
whose entries are defined by $\mathbf{A}_{ij}$ being the number of $[i]$ in $\varrho_{\mathbf{a}}([j])$. This matrix will become relevant later when we compute for statistical and geometric properties associated to $\varrho$. 
This construction can be seen as a generalization of \cite[Sec~4.4]{MRW}. This example is a specific case of such a substitution where one has $a_i=1$ for all $i\geq 0$.

\subsection{Compactification and its properties}

In order to use the results of \cite{MRW2}, the alphabet has to be compact in some topological space. Here, we use the boundedness of the entries of $\mathbf{a}$ to define an appropriate compactification of $\mathbb{N}_0$.

Consider the full shift $\mathcal S=\{*,0,1,2,\ldots,N\}^\Z$ over $N+2$ letters. This is a compact topological space under the local topology.  More specifically, for every $i\geq 0$, the $i$th non-trivial neighborhood (also called a cylinder set) of $\mathbf b$ in $\mathcal S$ consists of those elements of $\mathcal S$ that have the same subsequence $b_{-i}\ldots b_{-1}\fbox{$b_0$}b_1\ldots b_{i}$ around the origin. Here and in the sequel, the boxed term will always specify the location of the origin. 


It is well known that $\mathcal S$ is a metric space where the distance $d_\mathcal S(\mathbf b,\mathbf b')$ between $\mathbf b,\mathbf b'\in \mathcal S$ is given by $\frac{1}{2^{i}}$, where $i$ is the smallest number in  $\N_0$ such that the $i$th nontrivial neighborhoods of $\mathbf b$ and $\mathbf b'$ are different. In particular, if $\mathbf b$ and $\mathbf b'$ have different letters at the origin then $i=0$ and $d_\mathcal S(\mathbf b,\mathbf b')=1$.


Fix a sequence $\mathbf a\in \left\{0,1,\ldots,N\right\}^{\mathbb{N}_0}$. We now define an embedding $\iota_{\ab}\colon \N_0\to \mathcal{S}$ depending on $\ab$, which we do first by defining $\mathbf 0:=\iota_{\ab}([0])$. The image of $[0]$ is defined pointwise via the rule
\[
\iota_{\ab}([0])_m:=\begin{cases}
a_m,& \text{ for } m\geqslant 0, \\
*,& \text{otherwise}.
\end{cases}
\]
As a bi-infinite sequence, one has 
$$\mathbf 0:=\ldots ***\fbox{$a_0$}a_1a_2a_3\ldots\,.$$  For every $i>0$, we define $\iota_{\ab}([i]):=\sigma^{i}(\mathbf 0)$, i.e., 
$$\mathbf i:=\ldots ***a_0a_1\ldots a_{i-1}\fbox{$a_i$}a_{i+1}a_{i+2}\ldots\,.$$
Here, $\sigma$ is the left shift map as before. 

The closure $\mathcal{A}:=\overline{\iota_{\ab}(\N_0)} \subset \mathcal S$ is the alphabet that we use in our further construction. For an element $\mathbf b\in \mathcal{A}\setminus \iota_{\ab}(\N_0)$, we will denote the corresponding letter of the alphabet as $[\infty_\mathbf b]$. Note that we will always obtain at least one $[\infty_{\mathbf b}]$ from the compactification. Thus the total {\it compact alphabet} $\mathcal{A}$ can be seen as the union of all isolated letters $\iota_{\ab}([i])$ and all relevant infinities $[\infty_\mathbf b]$. For brevity, for every $i\in \N_0$, we identify the image $\iota_{\ab}([i])$ in the alphabet with $[i]$ in the pre-alphabet and use the notation $[i]$ for the isolated letters, so 
$$\mathcal{A}=\{[i]\mid i\in \N_0\}\cup \{[\infty_\mathbf b]\mid\mathbf b\in \mathcal{A}\setminus \iota_{\ab}(\N_0)\}.$$

Before defining the substitution on the letters $[\infty_\mathbf b]$, we prove two properties of the corresponding sequences in $\mathcal S$.

\begin{lem}\label{lem:nostar}
Let $\mathbf{b}\in \mathcal{A}$. 
Then $\mathbf b\in \mathcal{A}\setminus \iota_{\ab}(\N_0)$ if and only if no entry of $\mathbf b$ is $*$.
\end{lem}
\begin{proof}
This follows directly from the definition of $\iota_{\ab}(\N_0)$ as $\iota_{\ab}(\N_0)=\left\{\sigma^{n}(\mathbf{0}),n\in\mathbb{N}_0\right\}$. For every $n$, only the entries to the left of the $(-n)$th position of $\sigma^n(\mathbf 0)$ are $*$s. This means that accumulation points of $\iota_{\ab}(\N_0)$ are not in the $\sigma$-orbit of $\mathbf{0}$  and must be bi-infinite sequences over $\left\{0,1,\ldots,N \right\}$.
\end{proof}

\begin{lem}\label{lem:shift}
The shift map is invertible on the accumulation points. 
That is, if one has $\mathbf b\in \mathcal{A}\setminus \iota_{\ab}(\N_0)$, then $\sigma(\mathbf{b}),\sigma^{-1}(\mathbf{b})\in \mathcal{A}\setminus \iota_{\ab}(\N_0)$. 
\end{lem}

\begin{proof}
Since $\mathbf b$ is an accumulation point, there exists an increasing sequence $(i_n)_{n\geqslant 0}$ of natural numbers such that the sequence $(\mathbf{i}_n)_n$ with ${\bf i}_n:=\sigma^{i_n}({\bf 0})$ converges to $\bf{b}$ in the local topology. Without loss of generality, assume that $i_n>1$ for all $n$. Note that the sequence $\left(\sigma^{i_n-1}(\bf{0})\right)_{n\geqslant 0}$ converges to $\sigma^{-1}(\bf{b})$ and all elements of the sequence are in $\iota_{\ab}(\N_0)$ because they are shifts of $\bf{0}$, which implies $\sigma^{-1}(\bf{b})\in \overline{\iota_{\ab}(\N_0)}$. The claim for $\sigma({\bf{b}})$ can be proved using similar arguments. 
\end{proof}

We then extend the substitution in Definition~\ref{def:substitution} to the image of  all $[\infty_\mathbf b]$ under $\varrho^{ }_\mathbf a$. 
For $$\mathbf b=\ldots b_{-2}b_{-1}\fbox{$b_0$}b_1b_2\ldots \in \mathcal{A}\setminus \iota_{\ab}(\N_0),$$ 
we define
\begin{equation}\label{eq:subs-extend}
\varrho^{ }_\mathbf a([\infty_\mathbf b])=[0]^{b_0}[\infty^{ }_{\sigma^{-1}(\mathbf b)}][\infty^{ }_{\sigma(\mathbf b)}].
\end{equation}

Lemmas \ref{lem:nostar} and \ref{lem:shift} ensure that this substitution is well defined as $b_0$ must be a number from $\{0, 1,\ldots,N\}$ and not $*$, and that the letters $[\infty^{ }_{\sigma(\mathbf b)}]$ and $[\infty^{ }_{\sigma^{-1}(\mathbf b)}]$ exist in $\mathcal{A}$.


\begin{prop}
The map $\varrho^{ }_\mathbf a\colon \mathcal{A}\to \mathcal{A}^{+}$ defined in Def.~\ref{def:substitution} and Eq.~\eqref{eq:subs-extend} is continuous and hence a substitution. 
\end{prop}
\begin{proof}
 This follows immediately from the construction. 
In particular, if $d_\mathcal S(\sigma^{i}(\mathbf{0}),{\bf{b}})\leq \frac{1}{2^m}$ for some $m\geq 1$ and $i>0$, then $\mathbf i$ and $\mathbf b$ coincide at least within $m$ letters from the origin. Thus
$d_\mathcal S(\sigma^{i-1}(\mathbf{0}),\sigma^{-1}(\mathbf{b}))\leq\frac{1}{2^{m-1}}$. Similarly,  $d_\mathcal S(\sigma^{i+1}(\mathbf{0}),\sigma(\mathbf{b}))\leq \frac{1}{2^{m-1}}$. 
Additionally, the condition on $\sigma^{i}(\mathbf{0})$ implies $a_i=b_0$. 
This means that the words $\varrho^{ }_{\mathbf{a}}([i])=[0]^{a_i}[i-1][i+1]$ and  $\varrho^{ }_{\mathbf{a}}([\infty_{\mathbf{b}}])=[0]^{b_0}[\infty_{\sigma^{-1}(\mathbf{b})}][\infty_{\sigma(\mathbf{b})}]$ both belong to $\mathcal{A}^{b_0+2}$ and are close in $\mathcal{A}^{+}$ in the disjoint union topology.

The cases of other letters from $\mathcal A$ are similar.
\end{proof}

\begin{thm}\label{thm:primrecquasi}
Let $\bf{a}$ be a sequence of nonnegative integers satisfying conditions \textnormal{\ref{def:a1}--\ref{def:a3}} and $\varrho^{ }_{\bf{a}}$ be the substitution on  $\mathcal{A}$ corresponding to $\bf{a}$.
The substitution $\varrho^{ }_\mathbf a$ satisfies the following.
\begin{enumerate}[label={{\rm (\roman*)}}]
\item $\varrho^{ }_\mathbf a$ is primitive.
\item $\varrho^{ }_\mathbf a$ is recognizable. 
\item The associated (normalized) substitution operator is quasicompact. 
\end{enumerate}
\end{thm}
\begin{proof}

We first prove that every letter $a\in \mathcal{A}$ contains $[0]$ in its level-$(C+1)$ supertile $\varrho_\mathbf a^{C+1}(a)$, where $C$ is the constant from \ref{def:a3}. For this purpose, note that $\varrho^{ }_\mathbf a([0])$ contains $[0]$ because $a_0\neq 0$ due to \ref{def:a2}. Thus, if a supertile contains $[0]$, then every further application of $\varrho^{ }_\mathbf a$ to that supertile will contain $[0]$ as well.

If $a\in \iota_{\ab}(\N_0)$, then $a=[i]$ for some $i$. Among the integers $a_i,a_{i+1},\ldots,a_{i+C}$, at least one is positive due to condition \ref{def:a3}. Suppose $a_{i+k}>0$, where $0\leq k \leq C$. Then $\varrho_\mathbf a^{k}([i])$ contains the letter $[i+k]$. Since $\varrho_\mathbf a^{k+1}([i])$ contains $[0]$, so does $\varrho_\mathbf a^{C+1}([i])$.

If $a\in \mathcal{A}\setminus \iota_{\ab}(\N_0)$, then $a=[\infty_\mathbf b]$ for an appropriate $\mathbf b=\ldots b_{-2}b_{-1}\fbox{$b_0$}b_1b_2\ldots$. Similarly, among the numbers $b_0,b_1,\ldots,b_C$ at least one is positive because $\mathbf b$ is a limit of some subsequence of $(\sigma^i(\mathbf 0))_i$. After that, the arguments are similar to the previous case as $\varrho_\mathbf a^{k}([\infty_\mathbf b])$ contains $[\infty_{\sigma^k(\mathbf b)}]$.


We proceed with the proofs for the three claims in the theorem.

For primitivity, we need to show that, for every non-empty open set $U$, there exists $k(U)\in\mathbb{N}$ such that for every letter $a\in \mathcal{A}$, $\varrho_{\bf{a}}^{k(U)}(a)$ contains a letter in $U$. 
We consider two types of open sets: the singletons consisting of isolated points $[j],j\in\N_0$, and balls $B_{\varepsilon}([\infty_{\bf{b}}])$ for some accumulation point $\bf{b}$. 

If $[j]$ is an isolated point, the claim  is straightforward for $U=\left\{[j]\right\}$ since $[j]\triangleleft\varrho^{j}_{\bf{a}}([0])$ and hence $[j]\triangleleft\varrho^{j+C+1}_{\bf{a}}(a)$ for any $a\in \mathcal{A}$.

Now let $\bf{b}$ be an accumulation point and $U=B_{\varepsilon}([\infty_{\bf{b}}])$ for some $\varepsilon>0$. Following the proof of Lemma~\ref{lem:shift}, there exists a sequence of isolated points $\left\{{\bf{i}}_{n}\right\}_{n\geqslant 0}$ such that for every $\varepsilon>0$, there is a $T$ such that $d_\mathcal S({\bf{i}}_n,\bf{b})<\varepsilon$ for all $n>T$. For the embedding, this means $[i_n]\in U$. Same as above, we pick any $n>T$ and then for every $a\in\mathcal{A}$, we get that $[i_n]\in\varrho_{\bf{a}}^{i_n+C+1}(a)$, which proves the claim.

Next, we show that the substitution $\varrho^{ }_{\bf{a}}$ is recognizable. Here, we use three facts: 
\begin{enumerate}
        \item Zeros are always at the start of a supertile. (However, there may be level-$k$ supertiles without zeros for some $k\geq 1$.)
        \item  The supertiles $\varrho_{\bf{a}}([0])$ and $\varrho_{\bf{a}}([1])$ contain exactly one non-zero letter each.
        \item  Any supertile other than $\varrho_{\bf{a}}([0])$ and $\varrho_{\bf{a}}([1])$  contains exactly two non-zero letters.
\end{enumerate}
  In particular, all three facts together imply 
 \begin{enumerate}
     \item Runs of consecutive zeros belong to one single level-1 supertile.
     \item A supertile without zeros has exactly two letters.
 \end{enumerate}
  Thus, we can proceed as follows in order to determine level-1 supertiles uniquely:
whenever we see a (maximal) block of zeros, then a level-1 supertile starts with that block and the whole block belongs to that one supertile. 

Given $x\in X_{\varrho^{ }_{\bf a}}$, we show that there is a unique way to decompose it as a  concatenation of level-1 supertiles. First, we look to the right of the origin and look for the minimal $n\in\mathbb{N}$ such that $x_{n-1}\neq [0]$ and $x_n=[0]$. This means $x_n$ is a beginning of a level-1 supertile. We iterate the process to the right and to the left of $x_n$, placing a cut each time we are in such a situation. Between each two consecutive cuts we have a block of zeros followed by some number of non-zero letters. Since every level-1 supertile without initial zeros has exactly two non-zero letters, the initial block of zeros form a level-1 supertile with one or two non-zero letters (depending on the parity of the number of non-zero letters between two consecutive cuts). From this, one gets the unique decomposition into level-1 supertiles. 

The process for level-$k$ decomposition for $k>1$ is similar, where now one has to look at positions where $x_{[n,n+|\varrho^{k-1}_{\bf{a}}([0])|-1]}=\varrho^{k-1}_{\bf{a}}([0])$ and $x_{[n-|\varrho^{k-1}_{\bf{a}}([0])|,n-1]}\neq\varrho^{k-1}_{\bf{a}}([0])$. Since the substitution is injective, the subwords between every two consecutive positions that have this property can be unambiguously cut into level-$k$ supertiles. 

Lastly, we show that the associated (normalized) substitution operator is quasicompact. Here, we want to show that $\varrho^{ }_{\bf{a}}$ satisfies the conditions of Proposition~\ref{prop:quasicompact}. We show that this is satisfied for $k=C+2$, where $C$ is the constant from property \ref{def:a3} and for $F=\{[0],[1],\ldots,[C+1]\}$. 

We look at the images of letters in $\mathcal{A}$ under $\varrho^{C+2}_{\bf{a}}$. Every application of $\varrho^{ }_{\bf{a}}$ at least doubles the number of letters and therefore for all $a\in\mathcal{A}$, $|\varrho^{C+2}_{\bf{a}}(a)|\geq 2^{C+2}$. However, if we apply $\varrho^{ }_{\bf{a}}$ that many times, then at least once one letter will be substituted with three or more letters. Indeed, if $a=[j]$ for some $j\geq 0$, then after first $C+1$ applications of $\varrho_{\bf{a}}$, we will see the letters $[j+1], [j+2], \ldots, [j+C+1]$ and at least one of $a_{j+1}, a_{j+2}, \ldots, a_{j+C+1}$ must be positive by property \ref{def:a3}. The corresponding letter is substituted by at least three letters. The case for $a=[\infty_\mathbf b]$ is similar. This implies $|\varrho^{C+2}_{\bf{a}}(a)|\geq 2^{C+2}+1$ and by Lemma \ref{lem:spec-rad}, $r(M)^{C+2}\geq 2^{C+2}+1$.

On the other hand, we note that for every $n\leq C+1$, the supertile $\varrho^{n}_{\bf{a}}([0])$ only contains letters from $F$, and the supertile $\varrho^{C+2}_{\bf{a}}([0])$ contains only one letter not from $F$. For every other letter $a\neq [0]$, we will repeatedly apply $\varrho^{ }_{\bf{a}}$ $C+2$ times to $a$ but we will color some letters with red in the process. The goal is to show that red letters cannot lead to letters outside of $F$ and there are not too many non-red letters.

Here are the rules for how we color letters with red inductively. The initial letter $a$ is not colored with red. If one applies $\rho^{ }_{\bf{a}}$ to a non-red letter $[i]$ (or appropriate infinity), then the initial zeros of $\rho^{ }_{\bf{a}}([i])$, if any, are colored with red. Additionally, if we apply $\rho^{ }_{\bf{a}}$ to a red letter, then all resulting letters are red. For the sake of lucidity we illustrate the coloring process by an example in Table \ref{tab:ex-red}.

\begin{table}
\[ \mathbf{a}=1,0,1,0,1,0,1,0,1,0,1,0,\ldots\]
\[ \begin{array}{cll}
0 & \mapsto & 01\\
1 & \mapsto & 02\\
2 & \mapsto & 013\\
3 & \mapsto & 24\\
4 & \mapsto & 035\\
5 & \mapsto & 46\\
 & \vdots & \\
 \end{array} \]
For this choice of $\mathbf{a}$, $C=2$ (longest run of 0s is 1), and hence $k=4$. It is clear that $| \varrho^4(a)|>2^4$ for any letter $a$.
Choose $F=\{ [0], [1], [2], [3] \}$.  Four ($k=4$) iterations of $\varrho$ applied to some letter (here: 5) successively yield
\[ 5 \mapsto\ 46 \mapsto \underline{0}35\underline{0}57 \mapsto \underline{0}\underline{1}
2446\underline{0}\underline{1}4668 \mapsto \underline{0}\underline{1}\underline{0}\underline{2}\underline{0}
13\underline{0}35\underline{0}57\underline{0}\underline{1}\underline{0}\underline{2}\underline{0}35\underline{0}57\underline{0}57\underline{0}79\]

\caption{An example of the coloring process in the proof of Theorem \ref{thm:primrecquasi}.
Red letters are underlined. After four iterations of $\varrho$ there remain only twelve unmarked letters
($12\leqslant 2^4$), thus at most twelve letters managed to escape $F$. In fact, only nine out of the twelve unmarked letters --- $5,5,7,5,7,5,7,7,9$ --- actually escaped $F$.
\label{tab:ex-red}}
\end{table}

Note that no red zero can create a letter outside of $F$ because we apply the substitution to red zeros at most $C+1$ times. Also, all red letters originate from red zeros so all letters outside of $F$ will be non-red in the end. But the number of non-red letters at most doubles because every letter is substituted with at most two non-zero letters. Thus, for every isolated point $[i]$, 
\[
\left|\left\{[j]\in\varrho^{C+2}_{\bf{a}}([i])\colon [j]\not\in F\right\} \right| \leqslant 2^{C+2}<2^{C+2}+1\leqslant r(M)^{C+2}, 
\]
and a similar inequality holds for all infinities in $\mathcal A$ due to the continuity of $\varrho_{\bf{a}}$. By Lemma \ref{prop:quasicompact}, this implies the quasicompactness of $M$ (after normalization).
%
\end{proof}

\begin{proof}[Proof of Theorem \ref{thm:dyn}]
According to Theorem \ref{thm:primrecquasi}, the substitution $\varrho_{\mathbf a}$ is primitive 
and the corresponding normalized substitution operator is quasicompact. Applying Theorem~\ref{thm:unique-ergod} to our substitution implies Theorem \ref{thm:dyn}.
\end{proof}


\section{Inflation factor} \label{sec:infl}


The main goal of this section is to give an explicit formula for the inflation factor of the substitution $\varrho^{  }_{\bf{a}}$ provided $\mathbf a$ satisfies conditions \ref{def:a1}--\ref{def:a3}. This is done in Proposition~\ref{prop:length} below. First, we introduce an auxiliary parameter $\mu:=\mu(\mathbf a)$ which is the unique number in the interval $(0,1)$ that satisfies

\begin{equation}\label{eq:mu}
\frac{1}{\mu}=\sum_{i=0}^\infty a_i \mu^i.
\end{equation}

Existence and uniqueness follow from the observation that both parts of the equality are continuous as functions of $\mu$ on $(0,1)$ and the left-hand side is decreasing from $\infty$ to $0$ while the right-hand side is increasing from $a_0$ to $\infty$.
 
With this $\mu$, we set $$\lambda:=\mu+\frac{1}{\mu}>2.$$
Our first goal is to show that we can find an appropriate sequence $\mathbf a$ to get every possible~$\lambda$.

\begin{lem}\label{lem:inflation}
For every prescribed $\lambda>2$ we can find a suitable sequence $\mathbf a$ satisfying conditions \textnormal{\ref{def:a1}--\ref{def:a3}} such that $\lambda=\mu+\frac{1}{\mu}$ and $\mu$ satisfies Eq.~\eqref{eq:mu}.
\end{lem}

\begin{proof}
Let $0< \mu < 1$ be such that $\mu+\dfrac{1}{\mu}=\lambda$. 
First, we choose a positive integer $C$ such that $\mu+\mu^C\leq 1$. Then
$$\frac{1}{\mu}\geq \frac{1}{1-\mu^C}=\sum_{i=0}^\infty\mu^{iC}.$$
Let $\mu':=\frac{1}{\mu}-\frac{1}{1-\mu^C}\geq 0$ be the difference. We will show that it is possible to represent $\mu'$ as a sum of powers of $\mu$ with bounded nonnegative coefficients. The case of $\mu'=0$ is obvious as we can use zeros as the coefficients so we treat only the case $\mu'>0$ below.

Let $N$ be a positive integer such that $$\frac{1}{\mu}\leq \dfrac{N}{1-\mu}=\sum_{i=0}^\infty N\mu^i.$$
We consider the following family $F^0$ consisting of sequences of integers from $\left\{0,1,\ldots,N\right\}$ whose associated $\mu$-series sum is not greater than $\mu'$, i.e.,
$$F^0:=\left\{(b_i)_i\in \{0,1,\ldots,N\}^{\N_0} \colon \sum_{i=0}^\infty b_i\mu^i\leq \mu' \right\}.$$
The family $F^0$ contains infinitely many sequences. We recursively choose the sequence of subfamilies
$$F^0 \supseteq F^1 \supseteq F^2 \ldots$$ such that the $i$th term of every sequence in $F^{i+1}$ is the maximal number $c_i$ in $\{0,1,\ldots,N\}$ with the condition that $F^i$ contains infinitely many sequences with the $i$th term equal to $c_i$. 

We claim that ${\displaystyle \mu'=\sum_{i=0}^{\infty}c_i\mu^i}$. Indeed, the right-hand side cannot be greater than $\mu'$ because for every $i$, the sequence $(c_0,c_1,\ldots,c_i,0,0,\ldots)$ belongs to $F^i$.  

If ${\displaystyle \mu'>\sum_{i=0}^{\infty}c_i\mu^i}$,  we consider two cases. If there are infinitely many $c_i$s that are less than $N$, then for one sufficiently large $i$, we can swap $c_i$ with $c_i+1$ and keep the sum of the series less than $\mu'$. This contradicts the choice of $c_i$ since, in that case, there are infinitely many sequences in $F^i$ with $i$th term equal to $c_i+1 \leq N$.

In the second case, the sequence $(c_i)_i$ stabilizes at the value $N$. Let $j$ be the index such that
$$c_j<N=c_{j+1}=c_{j+2}=\ldots.$$ We claim that the sequence $(c_0,c_1,\ldots,c_{j-1},c_j+1,0,0,\ldots)$ gives the series with sum less than $\mu'$ and therefore $c_j$ was not chosen to be maximal. Indeed, by our assumption ${\displaystyle \mu'>\sum_{i=0}^{\infty}c_i\mu^i=\sum_{i=0}^{j}c_i\mu^i+\sum_{i=j+1}^{\infty}N\mu^i}$. By the choice of $N$, ${\displaystyle \mu^j\leq \sum_{i=j+1}^{\infty}N\mu^i}$. Combining these two inequalities, we get the claim and a contradiction with our assumption that $\mu'$ is greater than the sum of the series.

To finish the proof, we write
$$\frac{1}{\mu}=\sum_{i=0}^\infty\mu^{iC}+\mu'=\sum_{i=0}^\infty\mu^{iC}+\sum_{i=0}^\infty c_i\mu^i.$$
The first series ensures that properties \ref{def:a2} and \ref{def:a3} hold, and the second series guarantees property \ref{def:a1} for the sum.
\end{proof}

In the previous section, we have established that we can apply Theorem \ref{thm:primrecquasi} to our substitution.
In particular, it means there exists a natural tile length function associated with $\varrho^{ }_{\mathbf{a}}$, which is strictly positive.
Next, we show that $\lambda=r(M)$ and give explicit natural lengths for all letters in the subalphabet $\iota_{\ab}(\N_0)=\left\{[i]\colon i\in \N_0\right\}$ of isolated points. Recall that
$$\mathbf A=\left( 
\begin{array}{cccccc}
a_0& a_1+1& a_2 & a_3 & a_4 & \dots\\
1 & 0 & 1 & 0 & 0 & \ldots\\
0 & 1 & 0 & 1 & 0 & \ldots\\
0 & 0 & 1 & 0 & 1 & \ldots\\
0 & 0 & 0 & 1 & 0 & \ldots\\
\vdots & \vdots & \vdots & \vdots & \vdots & \ddots
\end{array}
\right).$$
Define $\ell:\iota_{\ab}(\N_0)\longrightarrow \R^+$ to be $\ell([0]):=1$, and for every $k>0$
\begin{equation}\label{eq:length}
\ell([k])=\mu^k+\sum_{j=1}^k\sum_{i=j}^\infty a_i\mu^{i+k+1-2j}.
\end{equation}

\begin{lem}\label{lem:unifcon}
The function $\ell\/\colon\/\iota_{\ab}(\N_0)\longrightarrow \R^+$ in Eq.~\eqref{eq:length} is uniformly continuous on\/ $\iota_{\ab}(\N_0)$ with respect to the topology on\/ $\mathcal{S}$.
\end{lem}

\begin{proof}
Let $[k],[k+t]\in \iota_{\ab}(\N_0)$ be such that $t>0$ and 
$d([k],[k+t])<\frac{1}{2^{n}}$ for some $n\in\N$. We show that $|\ell([k])-\ell([k+t])|<\varepsilon(n)$. First, note that $\sigma^{k}(\mathbf{0})$ and $\sigma^{k+t}(\mathbf{0})$ agree on $[-n,n]$ as the distance between these points in $\mathcal S$ is at most $\frac{1}{2^{n+1}}$. This means 
\begin{equation}\label{eq:n-close}
a_{k+r}=a_{k+t+r} \quad\text{for}\quad -n\leqslant r\leqslant n.
\end{equation} 
Consequently $n\leqslant k$ as $\sigma^k(\mathbf{0})$ has $*$ at the position $-k-1$, but all further iterations of $\sigma$ have non-$*$ entries there.

We can then split the double sum on the right hand-side of Eq.~\eqref{eq:length} into three parts:
\[
\ell([k])=\mu^k+\underbrace{\sum_{j=1}^{k-n}\sum_{i=j}^\infty a_i\mu^{i+k+1-2j}}_{\textnormal{(I)}}+\underbrace{\sum_{j=k-n+1}^k\sum_{i=k+n+1}^\infty a_i\mu^{i+k+1-2j}}_{\textnormal{(II)}}+\underbrace{\sum_{j=k-n+1}^k\sum_{i=j}^{k+n} a_i\mu^{i+k+1-2j}}_{\textnormal{(III)}}.
\]
Note that the component (III) is the same for both $\ell([k])$ and $\ell([k+t])$ because of Eq.~\eqref{eq:n-close}, and hence this component vanishes in $|\ell([k])-\ell([k+t])|$.
We now determine $n$-dependent bounds for (I). 
\begin{align*}
\sum_{j=1}^{k-n}\sum_{i=j}^\infty a_i\mu^{i+k+1-2j}& \leqslant N \sum_{j=1}^{k-n}\sum_{i=j}^\infty \mu^{i+k+1-2j}= N \sum_{j=1}^{k-n}\sum_{i=0}^\infty \mu^{i+k+1-j}\\
&=N \sum_{j=1}^{k-n} \frac{\mu^{k+1-j}}{1-\mu}=\frac{N}{(1-\mu)}(\mu^{n+1}+\cdots+ \mu^k)\leqslant \frac{N\mu^n}{(1-\mu)^2}.
\end{align*}
Here, the first inequality follows from \ref{def:a1}. One can carry out an analogous calculation for (II), which yields
\[
\sum_{j=k-n+1}^{k}\sum_{i=k+n+1}^\infty a_i\mu^{i+k+1-2j}\leqslant \frac{N\mu^{n+2}}{(1-\mu)^2}.
\]
Combining these estimates with the fact that $\mu^{k}-\mu^{k+t}< \mu^n$ gives us
\[
\left|\ell([k])-\ell([k+t])\right|<\left(1+\frac{2N}{(1-\mu)^2}+\frac{2\mu^2 N}{(1-\mu)^2}\right)\mu^n,
\]
from which the claim is immediate as the right-hand side goes to 0 as $n$ grows.
\end{proof}

Recall that $\mathcal{K}$ is the positive cone of $C(\mathcal{A})$.

\begin{prop}\label{prop:length}
The function $\ell$ in Eq.~\eqref{eq:length}
is a left eigenvector of $\mathbf{A}$ corresponding to $\lambda$. Moreover, it extends to a strictly positive, continuous natural length function on $\mathcal{A}$ and $\lambda=r(M)$ is the associated inflation factor. 
\end{prop}

\begin{proof}
We first show that the (infinite) vector $\ell_\mathcal A:=(\ell([0]),\ell([1]),\ell([2]),\ldots)$ of natural tile lengths is a left eigenvector of the matrix $\mathbf A$ corresponding to eigenvalue $\lambda=\mu+\frac{1}{\mu}$, so $\ell_\mathcal A\mathbf A=\lambda \ell_\mathcal A$.

We consider the product in the left-hand side column by column. Particularly, the first column yields 
$a_0\ell([0])+\ell([1])=\lambda \ell([0])$. Since $\ell([0])=1$, one has
$$\ell([1])=\lambda -a_0=\mu+\frac{1}{\mu}-a_0=\mu+\sum_{i=0}^\infty a_i\mu^i-a_0=\mu+\sum_{i=1}^\infty a_i\mu^i,$$
which satisfies the statement of this proposition.

After that, looking at the $k$th column for $k>1$ we get
$$a_{k-1}\ell([0])+\ell([k-2])+\ell([k])=\lambda \ell([k-1]),$$
or
$$\ell([k])=\lambda \ell([k-1])-\ell([k-2])-a_{k-1}.$$

In particular, for $k=2$,
\begin{multline*}\ell([2])=\lambda \ell([1])-\ell([0])-a_1=\left(\mu+\frac{1}{\mu}\right)\left( \mu+\sum_{i=1}^\infty a_i\mu^i\right)-1-a_1=\\
=\mu^2+1+\sum_{i=1}^\infty a_i\mu^{i+1}+\sum_{i=1}^\infty a_i\mu^{i-1}-1-a_1=\\=
\mu^2+\sum_{i=1}^\infty a_i\mu^{i+1}+\sum_{i=2}^\infty a_i\mu^{i-1}=\mu^2+\sum_{j=1}^2\sum_{i=j}^\infty a_i\mu^{i+3-2j}.
\end{multline*}

For all other $k$, the computations can be carried out by induction. Using $\lambda=\mu+\frac{1}{\mu}$ and expanding the product $\lambda\ell([k-1])$ we get
\begin{multline*}\ell([k])=\lambda \ell([k-1])-\ell([k-2])-a_{k-1}=\\=
\left(\mu+\frac{1}{\mu}\right)\left(\mu^{k-1}+\sum_{j=1}^{k-1}\sum_{i=j}^\infty a_i\mu^{i+k-2j}\right)-\left(\mu^{k-2}+\sum_{j=1}^{k-2}\sum_{i=j}^\infty a_i\mu^{i+k-1-2j}\right)-a_{k-1}=\\=
\mu^k+\mu^{k-2}+\sum_{j=1}^{k-1}\sum_{i=j}^\infty a_i\mu^{i+k+1-2j}+\sum_{j=1}^{k-1}\sum_{i=j}^\infty a_i\mu^{i+k-1-2j}-\mu^{k-2}-\sum_{j=1}^{k-2}\sum_{i=j}^\infty a_i\mu^{i+k-1-2j}-a_{k-1}.
\end{multline*}

Here, we can cancel $\mu^{k-2}$ and also cancel the double sum with negative sign leaving only the term with $j=k-1$ from the second double sum. Simplifying, the expression, we get
$$=
\mu^k+\sum_{j=1}^{k-1}\sum_{i=j}^\infty a_i\mu^{i+k+1-2j}+\sum_{i=k-1}^\infty a_i\mu^{i+k-1-2(k-1)}-a_{k-1}.$$
Note that the first term of the second sum is $a_{k-1}$ so we can cancel that as well. Additionally, the power of $\mu$ in the second sum is $i+k-1-2(k-1)=i+k+1-2k$. These adjustments give us the final expression
$$=\mu^k+\sum_{j=1}^{k-1}\sum_{i=j}^\infty a_i\mu^{i+k+1-2j}+\sum_{i=k}^\infty a_i\mu^{i+k+1-2k}=\mu^k+\sum_{j=1}^k\sum_{i=j}^\infty a_i\mu^{i+k+1-2j}$$
as needed.

We now show that $\lambda=r(M)$. Since $\rho^{ }_{\mathbf a}$ is primitive and
$M/r(M)$ is quasicompact, the spectral radius $r(M)$ of the substitution operator $M$ is an eigenvalue with a strictly positive, continuous eigenfunction $\ell$;  see \cite[Thm. 5.4(B)]{MRW2} and  \cite[Thm.~V.5.2 and Prop.~V.5.6]{Sch}. Moreover, $r(M)$ is the only eigenvalue of $M$ with an eigenfunction in $\mathcal{K}$; see \cite[Thm.~V.5.2(iv)]{Sch}.  

Since the subalphabet $\iota_{\mathbf{a}}(\N_0)$ is dense in $\mathcal{A}$ and since $\ell$ is uniformly continuous  on $\iota_{\mathbf{a}}(\N_0)$ by Lemma~\ref{lem:unifcon}, it follows that $\ell$ extends to a unique continuous function on $\mathcal{A}$, which must necessarily be in the positive cone $\mathcal{K}$, and must be an eigenfunction of $M$ corresponding to the eigenvalue $\lambda$. It then follows that $\lambda=\mu+\frac{1}{\mu}$ must be the spectral radius $r(M)$ of the substitution operator. 
\end{proof}

\begin{rem}
The second claim in the previous result is an infinite-alphabet analogue of the Perron eigenvalue condition by Lagarias and Wang for inflation functional equations admitting weak Delone set solutions in $\mathbb{R}^{d}$, which states $\det(A)=\lambda(S)$ where $\lambda(S)$ is the Perron--Frobenius eigenvalue of the subdivision matrix $S$ and $A$ is the inflation map; see~\cite{LW}. 
\end{rem}

Our next goal is to prove Theorem \ref{thm:allfactors}. We start by giving the formal definition of a Delone set, which is needed in the theorem. A point set $\varLambda$ in $\mathbb{R}^d$ is called a Delone set if it is \begin{itemize}
\item uniformly discrete, i.e., there is a positive radius $r$ such that the balls of radius $r$ centered at points of $\varLambda$ form a packing in $\R^d$, and  
\item relatively dense, i.e., there is a positive radius $R$ such that the balls of radius $R$ centered at points of $\varLambda$ form a covering of $\R^d$.
\end{itemize}

\begin{proof}[Proof of Theorem~\textnormal{\ref{thm:allfactors}}]
For a given $\lambda$, we can find an appropriate sequence using the results of Lemma \ref{lem:inflation}. The resulting alphabet is infinite and compact. It remains to show that every element of the subshift gives rise to a Delone set. This follows automatically from the continuity and strict positivity of $\ell$, which means there exist constants $C_1,C_2>0$ such that for every $a\in \mathcal{A}$, $C_1\leq \ell(a)\leq C_2$ where $\ell$ is the natural length function. The lower bound implies that the point set is uniformly discrete, while the upper bound guarantees that it is relatively dense. 


In fact, we get an explicit lower bound from  property \ref{def:a3}, which states that for every $k>0$, at least one of the coefficients $a_k,\ldots,a_{k+C}$ is not zero. This means the internal sum in Eq.~\eqref{eq:length} for $j=k$ is at least 
$$a_k\mu+a_{k+1}\mu^2+\ldots+a_{k+C}\mu^{C+1}.$$
Thus for every letter $[k], k\in\N_0$, $\ell([k])\geq \mu^{C+1}$. 
This completes the proof.
\end{proof}



\begin{rem}
It is worth noting that the substitution on the alphabet $\mathcal{A}$ is always on infinitely many letters but the resulting Delone set may have finite local (geometric) complexity.
For example, if $a_0=2$ and $a_1=a_2=a_3=\ldots=1$, then $\varrho_\mathbf a$ is a constant-length substitution with $\lambda =3$. In this case, the natural lengths of all tiles in $\mathcal{A}$ are $1$ and the resulting Delone set coincides with $\mathbb Z$.

Moreover, if $\mathbf a$ is such that $\lambda$ is an integer, then the lengths of all tiles are integers. This can be seen from the proof of Proposition~\ref{prop:length} for tiles $[k]$ and from continuity of the natural length function for the accumulation points of the alphabet. Therefore, the resulting Delone set is a subset of $\Z$ and is of finite local complexity as well.
\end{rem}

\begin{exam} For every suitable sequence $\mathbf a$, we can construct a Delone set $\varLambda$ with the corresponding geometric inflation symmetry defined by $\lambda$ or its integer power even when $\lambda$ is transcendental.

Let $k$ be an integer such that $\rho^k_{\mathbf a}([0])$ contains the letter $[0]$ in its interior. Once we turn the word  $\rho^k_{\mathbf a}([0])$ into a patch of length $\lambda^k\ell([0])=\lambda^k$ there is an internal segment corresponding to $[0]$ of length 1. If $a_0=1$, we can take $k=2$ and if $a_0>1$, then we can choose $k=1$. We then place the origin strictly inside this unit segment so that the $\lambda^k$-inflation turns the unit segment into the segment of length $\lambda^k$ that correspond to $\rho^k_{\mathbf a}([0])$.
With this choice of the origin the substitution $\rho^k_{\mathbf a}$ applied to $[0]$ repeatedly will give a Delone set $\varLambda$ such that $\lambda^k \varLambda\subset \varLambda$ even when $\lambda$ (and hence $\lambda^k$) is transcendental. 

Note that this does not require one to start with a bi-infinite symbolic fixed point of $\rho$ (which in general need not exist).
Moreover, due to Theorem~\ref{thm:dyn}, the tiling associated to this special Delone set $\varLambda$ generates the geometric hull $(\Omega_\rho,\mathbb R)$ by minimality.
\end{exam}

We next look at (relative) frequencies of tiles in $\mathcal{T}\in \Omega_{\varrho}$ and relative frequencies of letters for elements of the subshift. To do this, we consider the \emph{dual operator} $M^{\prime}$ of the substitution operator. Here $M^{\prime}$ acts on $ C(\mathcal{A})^{\prime}$, where $C(\mathcal{A})^{\prime}$ is the space of continuous real-valued linear functionals on $C(\mathcal{A})$. 
Thus, right eigenvectors of $\mathbf{A}$ are related to the eigenfunctions of the dual operator $M^{\prime}$. We have the following useful observation.

\begin{lem}\label{lem:right-EV}
Let $\mathbf d = (1,\mu,\mu^2,...)^T$ and $\lambda =\mu +\dfrac 1 \mu$, then $$\mathbf {Ad}=\lambda \mathbf d.$$
\end{lem}

\begin{rem}
Here $(\dots)^T$ denotes the transpose vector (also infinite) in order to make it compatible with the matrix-vector product.
\end{rem}

\begin{proof}

Writing the equation $\mathbf {Ad}=\lambda \mathbf d$ row by row, we obtain 
\[ \mu + \sum_{i=0}^\infty a_i \mu^i =  \mu + \frac{1}{\mu}=\lambda\cdot 1 \]
from the first row. This equation is satisfied by Eq.~\eqref{eq:mu}. 
The $i$th row (for $i \geq 2$) is 
\[ \mu^{i-2}+\mu^i = (\mu + \mu^{-1}) \mu^{i-1}=\lambda \mu^{i-1} , \]
and the claim follows.
\end{proof}

We now get the following consequence regarding relative frequencies of isolated points for any element $x$ in the subshift $X_{\varrho^{ }_{\bf a}}$ generated by $\varrho^{ }_{\mathbf{a}}$.

\begin{prop} 
Let $\mathbf a$ be a sequence that satisfies \textnormal{\ref{def:a1}--\ref{def:a3}}. Then the vector $(1-\mu)\mathbf d$ is the vector of (relative) frequencies of the isolated points and the frequency of every infinity $[\infty_\mathbf b]$ is $0$.
\end{prop}

\begin{proof}

From primitivity and quasicompactness, it follows that the dual operator $M^{\prime}$ also has a one-dimensional eigenspace corresponding to $\lambda=r(M)$, which is spanned by a strictly positive linear form $\phi$ in the dual space $C(\mathcal{A})^{\prime}$ \cite[Prop.~III.8.5(c)]{Sch}. Here, strict positivity means $\phi$ assigns a positive number to non-zero elements of the positive cone $\mathcal{K}$ of $C(\mathcal{A})$, i.e., $\phi(f)>0$ for $0\neq f\in \mathcal{K}$. 

Consider the functions of the form $\bbo_{[i]}\in \mathcal{K}$ given by $\bbo_{[i]}(a)=1$ if $a=[i]$ and $0$ otherwise, where $[i]\in\mathcal{A}$ is an isolated point. It follows that $\phi(\bbo_{[i]})>0$. By the Riesz--Markov representation theorem, $\phi$ corresponds to a unique regular Borel measure $\nu:=\nu^{(\phi)}$ on the alphabet $\mathcal{A}$. With respect to this identification, one has $\nu_i:=\phi(\bbo_{[i]})=\nu(\{[i]\})>0$. 

Since $\phi$ is an eigenfunctional of $M^{\prime}$, we get 
\[
\phi(M\bbo_{[i]})=(M^{\prime}\phi)(\bbo_{[i]})=\lambda\phi(\bbo_{[i]})=\lambda\nu_i. 
\]
Fix $[i]\neq [0]$. From the definition of $M$, we get 
\begin{align*}
    \phi(M\bbo_{[i]})&=\int_{\mathcal{A}} (M\bbo_{[i]})(a)\,\dd\nu(a)=\int_{\mathcal{A}} \sum_{b\,\triangleleft\, \varrho(a)}\bbo_{[i]}(b)\,\dd\nu(a)\\
    &=\nu_{i{-}1}\cdot \left(\sum_{b\,\triangleleft\, \varrho([i{-}1])}\bbo_{[i]}(b)\right) +\nu_{i{+}1}\cdot \left(\sum_{b\,\triangleleft\, \varrho([i{+}1])}\bbo_{[i]}(b)\right)\\
    &=\nu_{i{-}1}+\nu_{i{+}1}.
\end{align*}
Here, the third equality holds because the letter $[i]$ only appears in $\varrho([i{-}1])$ and $\varrho([i{+}1])$. 
For the letter $[0]$, we have the relation
\begin{equation}\label{eq:meas-zero}
\lambda \nu_0=\int_{\mathcal{A}} \sum_{b\,\triangleleft\, \varrho(a)}\bbo_{[0]}(b)\,\dd\nu(a)\geqslant \sum_{i\in \mathbb{N}} \nu_i\cdot\sum_{b\,\triangleleft\, \varrho([i])}\bbo_{[0]}(b)= \nu_1+\sum_{i=0}^{\infty} a_i\nu_i.
\end{equation}
Note that the values of $\nu_0$ and $\nu_1$ completely determine the values of  $\nu_i$ for all $i\geqslant 2$. Now choose $\frac{\nu_1}{\nu_0}=\mu$. This forces $\nu_i=\mu^i\nu_0$. From Lemma~\ref{lem:right-EV}, we get that 
\[
\lambda \nu_0= \nu_1+\sum_{i=0}^{\infty} a_i\nu_i, 
\]
which implies that the inequality in Eq.~\eqref{eq:meas-zero} is indeed an equality. In particular, this implies that the measure $\nu$ is atomic and only gives positive weights to isolated points, i.e., $\nu(\infty_{\mathbf{b}})=0$ for any accumulation point $\infty_{\mathbf{b}}$ . Otherwise, this would then define a new measure $\nu^{\prime}$ that necessarily corresponds to a \emph{different} eigenfunctional $\phi^{\prime}$, which is not a rescaled version of $\phi$, contradicting the uniqueness of $\phi$.

Now fix an isolated point $[i]\in\mathcal{A}$. 
Note that, by some abuse of notation, one can view $\bbo_{[i]}\in C(\mathcal{A})$ as a continuous function on the set of prototiles (which is indexed by $\mathcal{A}$).
Let $\mathfrak{t}_{[i]}$ be the tile associated to $[i]$ (with length $\ell([i])$). Using this, one builds the tiling dynamical system $(\Omega,\mathbb{R})$ from the subshift. 
From unique ergodicity, every tiling $\mathcal{T}\in \Omega$ is generic with respect to the unique invariant measure $\mu$, which is induced by the sequence $\{\nu_n\}_{n\geqslant 0}$ of volume-normalised, transition-consistent measures on the set $\mathcal{P}_n$ of level-$n$ supertiles; see \cite{FS} and \cite{MRW2} for details. Since we have recognizability, one can use this sequence of measures to calculate certain averages of (measurable) functions on level-$n$ supertiles. Here, $\nu_0=\frac{1}{\phi(\ell)}\nu$, where $\phi(\ell)$ is the volume normalisation constant. In particular, for the function $\bbo_{[i]}$, one has 
\begin{equation}\label{eq:freq-formula}
\lim_{R\to\infty}\frac{1}{2R+1} \#(\mathfrak{t}_{[i]}\text{ in } \mathcal{T}_{[-R,R]})=\int_{a\in \mathcal{A}} \bbo_{[i]}(a)\, \dd \nu_0(a)=\frac{1}{\phi(\ell)}\int_{a\in \mathcal{A}} \bbo_{[i]}(a)\, \dd \nu(a)=\frac{\phi(\bbo_{[i]})}{{\phi(\ell)}}.  
\end{equation}
Note that this value is the same as the \emph{abstract frequency} $\text{freq}_{\mu}(I)$  of the set $I=\{\mathfrak{t}_{[i]}\}$ with respect to $\mu$. One can derive this frequency from the ergodic theorem by considering the function $\chi_{[i],U}\in C(\varOmega)$ given by 
\[
\chi_{[i],U}(\mathcal{T})=\begin{cases}
    1,& \text{ if } \mathfrak{t}_{[i]} \text{ is at }\mathcal{T}-v, \text{ for some } v\in U,\\
    0, & \text{otherwise},
\end{cases}
\]
for any small enough open neighbourhood $U$ of $0$. This is nothing but the characteristic function of the set  
$\varOmega^{ }_{[i],U}$, which comprises of all tilings with tile $\mathfrak{t}_{[i]}$ at the origin (up to some small displacement $v\in U$). One then has
\[
\text{freq}_{\mu}(I)=\frac{\mu(\varOmega^{ }_{I,U})}{\text{vol}(U)}=\frac{\phi(\bbo_{[i]})}{\phi(\ell)}.
\]
The claim on the \emph{relative frequencies} follows from the appropriate normalisation so that the sum of all relative frequencies of isolated points equals $1$. 
\end{proof}

\begin{rem}
It is a non-trivial statement that the relative frequency of the tile $\mathfrak{t}_{[i]}$ in $\varOmega$ is the same as the (relative) letter frequency of $[i]$ for every element of the subshift $X_{\varrho}$ (which also exists uniformly due to unique ergodicity). To see this, one identifies $(\Omega,\mathbb{R})$ with suspension flow 
$(Y,\mathbb{R})$ of $(X_{\varrho},\sigma)$, with roof function $h\colon X_{\varrho}\to \mathbb{R}_{>0}$, $h(x)=\ell(x_0)$,
where $Y=(X_{\varrho}\times \mathbb{R})/\sim$, Here, the equivalence relation is given by the standard identification $(x,\ell(x_0))\sim (\sigma(x),0)$. 
 To compute for letter frequencies in $X_{\varrho}$, one exploits the one-to-one correspondence between the space $\mathcal{M}(X_{\rho})$ of $\sigma$-invariant measures on  $X_{\rho}$ and the space $\mathcal{M}(\Omega)$ of $\mathbb{R}$-invariant measures on $\Omega$. We refer to \cite{AK} for details on suspension flows and \cite{BBG} for relations between invariant measures and transverse measures.
\end{rem}

\begin{rem} 
The formula given in Eq.~\eqref{eq:freq-formula} admits a general version for frequencies of patches in $\mathcal{T}$; see \cite{FS}. This can be used to compute for relative frequencies of finite words in the subshift. 
Note that any legal word with $\infty_{\mathbf{b}}$ has frequency zero, so their corresponding supertiles also have frequency zero.
We then compute the abstract frequencies of the corresponding patches in $\mathcal{T}$ via the formula 
\[
\text{freq}_{\mu}(\mathcal{P})=\lim_{n\to \infty} \sum_{[i]\in \mathbb{N}_0} \#(\mathcal{P}\text{ in } \varrho^{n}(\mathfrak{t}_{[i]}))\,\frac{\dd\nu(\{[i]\})}{\lambda^n},
\]
where we denote by $\varrho^{n}(\mathfrak{t}_{[i]}))$ the level-$n$ (geometric) supertile corresponding to the tile $\mathfrak{t}_{[i]}$. For simple cases, this 
limit stabilises, giving an explicit formula in finitely many steps. The relative frequencies, (which in turn are also the relative frequencies of the corresponding words) can be obtained by finding the appropriate normalisation. 

Let $\widetilde{\mu}$ be the unique invariant measure on $X_{\rho_{\mathbf{a}}}$. One can then use the computations above to associate a measure to any neighbourhood of a legal word $w$ that contains an accumulation point. As an example, for 
$[\infty_{\mathbf{b}}][0]\in \mathcal{L}$, one has

\[
\widetilde{\mu}(B_{\varepsilon}(\left\{[\infty_{\mathbf{b}}][0]\right\}))=\frac{1}{C}\sum_{[i]\in B_{\varepsilon}(\infty_{\mathbf{b}})} \text{freq}_{\mu}(\{\mathfrak{t}_{[i]}\mathfrak{t}_{[0]}\}),\] 
for any accumulation point $\infty_{\mathbf{b}}$, where $C$ is a statistical normalisation constant.

\end{rem}

\section{Periodic sequences} \label{sec:periodic}


The main goal of this section is to show that if the sequence $\mathbf a$ is eventually periodic, then both the associated parameter $\mu$ and the inflation factor $\lambda$ are algebraic. The converse does not hold and we provide an explicit example of a rational inflation factor that cannot be achieved in this framework using an eventually periodic $\mathbf a$.

\begin{prop}
If $\mathbf a$ is eventually periodic, then $\mu$ and $\lambda$ are algebraic.
\end{prop}

\begin{proof}
We assume that the sequence $(a_i)_i$ is periodic starting with index $j$ with period $k$. So if $i\geq j$, then $a_{i+k}=a_i$. Then we can write
\begin{multline}\label{eq:algebraic}
\frac{1}{\mu}=\sum_{i=0}^\infty a_i\mu^i=a_0+\ldots+a_{j-1}\mu^{j-1}+\sum_{i=j}^\infty a_i\mu^i=\\=
a_0+\ldots+a_{j-1}\mu^{j-1}+(a_j\mu^{j}+\ldots+a_{j+k-1}\mu^{j+k-1})\sum_{m=0}^{\infty}\mu^{mk}=\\=a_0+\ldots+a_{j-1}\mu^{j-1}+\frac{a_j\mu^{j}+\ldots+a_{j+k-1}\mu^{j+k-1}}{1-\mu^k}.\end{multline}
Multiplying by $\mu(1-\mu^k)$ we get that $\mu$ is a root of a polynomial with integer coefficients and hence algebraic.
Since $\lambda=\mu+\frac{1}{\mu}$, it is algebraic as well.
\end{proof}

\begin{exam}
Let us take $\mu=\frac25$ and $\lambda = \frac{29}{10}$.

We note that Eq.~\eqref{eq:algebraic} results in a polynomial with integer coefficients and constant term~$1$. For the choice of $\mu$ above, no polynomial with integer coefficients and root $\mu$ can have constant term~$1$. This means that every sequence $\mathbf a$ that results in $\lambda = \frac{29}{10}$ as the inflation factor cannot be eventually periodic.
\end{exam}

Another feature that shows a much clearer distinction between periodic and non-periodic sequences $\mathbf a$ is the structure of the accumulation points in the alphabet. If $\mathbf a$ is eventually periodic, then one can establish that the set $\mathcal{A}\setminus \iota_{\ab}(\N_0)$ is finite. On the other hand, if $\mathbf a$ is not eventually periodic, the family of accumulation points is infinite and, in particular, can be uncountable.


%

\section{Example with transcendental inflation factor}\label{sec:trans}

In this section, we give an explicit substitution whose inflation factor $\lambda$ is transcendental. Consider the Thue--Morse sequence $\mathbf{t}=(t_n)_{n \geq 0}$ with $t_n:=(-1)^{s_2(n)}\in\left\{-1,1\right\}$, where $s_2(n)$ is the number of ones in the binary expansion of $n$, see \cite[\texttt{A010060}]{OEIS}. Consider the generating function $T(z):=\sum_{n\geqslant 0} t_nz^n$, which is a transcendental power series over $\mathbb{Q}(z)$; see \cite{Mahler}. Mahler proved the following result regarding values of $T$
on algebraic parameters. 

\begin{thm}[\cite{Mahler}]\label{thm:Mahler}
Let $\alpha\neq 0$ be an algebraic number with $|\alpha|<1$. Then the number $T(\alpha)$ is transcendental.
\end{thm}

We would like to leverage this result and use a modified version of $t_n$ as our sequence which will define a substitution. To this end, let 
\begin{equation}\label{eq:thue-morse}
\mathbf{a}=(a_n)_{n \geq 0} \,\,\mbox{ with } a_n:=(3+t_n)/2,
\end{equation}
which yields the (strictly positive) Thue--Morse sequence where $1$ is replaced by $2$ and $-1$ is replaced by $1$. The sequence obviously satisfy conditions \ref{def:a1}--\ref{def:a3} as it has no zeros and is bounded. One can easily verify that the generating function $A(z)$ for the sequence $a_n$ satisfies
\begin{equation}\label{eq:gen-function}
A(z)=\frac{3}{2}\cdot\frac{1}{1-z}+\frac{1}{2}T(z).
\end{equation}
We now have the following result. 

\begin{thm}\label{thm:TM-transc}
Let $\mathbf{a}$ be the sequence given in Eq.~\eqref{eq:thue-morse} and consider the corresponding substitution $\varrho^{ }_{\mathbf{a}}$. One then has 
\begin{enumerate}
\item[\textnormal{(i)}] The subshift $(X_{\varrho_{\bf a}},\sigma)$ and the tiling dynamical system $(\Omega_{\varrho_{\bf a}},\mathbb{R})$ are strictly ergodic. 
\item[\textnormal{(ii)}] There exists a unique natural tile length function for $\varrho_{\mathbf{a}}$ which is strictly positive.
\item[\textnormal{(iii)}] The inflation factor $\lambda$ corresponding to $\varrho_{\mathbf{a}}$ is transcendental.
\end{enumerate}
\end{thm}

\begin{proof}
The first two claims follow directly from Theorem~\ref{thm:dyn} since $\varrho_{\mathbf{a}}$ is primitive
and has a quasicompact substitution operator (after normalization). It remains to show that $\lambda$ is transcendental. 

We know that $\lambda=\mu+\frac{1}{\mu}$, where $\mu\in (0,1)$ is the unique solution to Eq.~\eqref{eq:mu}, which means $\frac{1}{\mu}=A(\mu)$. It suffices to show that $\mu$ is transcendental. From Eq.~\eqref{eq:gen-function}, one has 
\[
\frac{1}{\mu}=A(\mu)=\frac{3}{2}\cdot\frac{1}{1-\mu}+\frac{1}{2}T(\mu). 
\]
Suppose on the contrary $\mu$ is algebraic. Then, by Theorem \ref{thm:Mahler}, $T(\mu)$ is transcendental because $\mu\in (0,1)$.

On the other hand, $T(\mu)=\frac{2}{\mu}-\frac{3}{1-\mu}$ is algebraic which is a contradiction. Thus $\mu$ is transcendental and so is $\lambda$, as $\mu$ is the solution of the equation $x^{2}-\lambda x +1=0$.
%
\end{proof}

\begin{rem}
The substitution described in the previous theorem and a patch of the tiling it generates are available online at \cite{Enc} under the title ``Tiling with Transcendental Inflation Multiplier''. Here $\mathbf a=1,2,2,1,2,1,1,2,\ldots$. This particular sequence yields $\mu=0.3753\ldots$. The corresponding inflation factor is $\lambda=2.7899\ldots$. The natural lengths of the first few letters are
$$\ell([0])=1, \quad
\ell([1])=1.7899\ldots, \quad
\ell([2])=1.9937\ldots, \quad
\ell([3])=1.7724\ldots, \quad
\ell([4])=1.9510\ldots.$$
As an illustration, the third iteration of the substition applied to $[0]$ is
$$\varrho^3_{\mathbf a}([0])=[0][1][0][0][0][2][0][1][0][1][0][1][0][0][1][3].$$
\end{rem}

In fact, adding a nonnegative periodic sequence $\mathbf{b}=(b_n)_{n\geqslant 0}$ to $\mathbf{a}$ gives rise to another sequence whose corresponding $\lambda$  is also transcendental. To see this, suppose  the smallest period of $\mathbf{b}$ is $p$. The generating function for $\mathbf{b}$ is then \[B(z)=\frac{b_0+b_1 z+\cdots +b_{p-1}z^{p-1}}{1-z^p}\] where the $b_i$s are (nonnegative) integers. Considering $\mathbf{a}+\mathbf{b}$, and inserting $\mu$ in the generating function, we get
\[
\frac{1}{\mu}=(A+B)(\mu)=\frac{3}{2}\cdot\frac{1}{1-\mu}+\frac{1}{2}T(\mu)+\frac{b_0+b_1 \mu+\cdots +b_{p-1}\mu^{p-1}}{1-\mu^p}, 
\]
where the last term is clearly algebraic whenever $\mu$ is algebraic. We arrive at the same contradiction, and hence $\lambda$ must again be transcendental. The same argument extends when $\mathbf{b}$ is eventually periodic. This provides a way to construct an explicit infinite subfamily with transcendental $\lambda$.

\begin{rem}
We note that, in general, the transcendence of the generating function $A(z)$ over the field $\mathbb{Q}(z)$ does not suffice to conclude that $A(\mu)$ is a transcendental number whenever $\mu\in (0,1)$ is algebraic. 
The Thue--Morse power series $T(z)$ satisfies the functional equation $T(z)=(1-z)T(z^2)$. For power series satisfying similar functional equations (called Mahler-type functional equations), dichotomy results regarding the transcendence of numbers of the form $A(\mu)$ 
are available; see \cite[Thm.~5.3]{Adam}. We refer the reader to \cite{Adam} for a survey on Mahler's method on transcendence and linear independence results. 
\end{rem}

\section{Concluding remarks} \label{sec:remarks}

There is an alternative definition of quasicompactness which involves the \emph{essential spectrum}. For the substitution operator $M$, this is equivalent to the essential spectral radius $r_{\textnormal{ess}}(M)$ being strictly less than $r(M)$. We conjecture that $r_{\textnormal{ess}}(M)=2$ for all substitutions we constructed in this work which satisfy the conditions  \textnormal{\ref{def:a1}--\ref{def:a3}}. This has a variety of implications in discrepancy estimates, in particular to questions regarding bounded distance equivalence, which we explore in another work \cite{FGM}. 

It also remains to find the right conditions on the inflation factor $\lambda$ which guarantee or exclude
the existence of measurable or continuous eigenvalues for either dynamical system
$(X_{\varrho_{\bf a}},\sigma)$ or $(\Omega_{\varrho_{\bf a}},\mathbb{R})$. Note that convergence of points $x_n\to x$ in $X_{\varrho_{\bf a}}$ is more subtle than in shifts over finite alphabets, where converging points must necessarily agree on a large patch around the origin. In the infinite local complexity regime, there are more mechanisms for convergence which do not require exact agreement. There are results on weaker versions of the Pisot--Vijayaraghavan problem; see \cite{B2} for instance. It would be interesting to see whether such numbers with biased distribution mod $1$ give rise to substitutions with interesting spectral features.





\section*{Acknowledgments} 
We extend our gratitude to two anonymous reviewers for insightful comments that helped improve the paper. The authors would  also like to thank Michael Baake, Michael Coons, Dan Rust, and Jamie Walton for valuable discussions. In particular, we thank Michael Coons for references to Mahler's method.

\bibliographystyle{amsplain}


\begin{dajauthors}
\begin{authorinfo}[df]
  Dirk Frettl\"oh\\
  Bielefeld University\\
  Bielefeld, Germany\\
  dfrettloeh\imageat{}techfak\imagedot{}de \\
  \url{https://www.math.uni-bielefeld.de/~frettloe/}
\end{authorinfo}
\begin{authorinfo}[ag]
  Alexey Garber\\
  University of Texas Rio Grande Valley\\
  Brownsville, TX, USA\\
  alexey\imagedot{}garber\imageat{}utrgv\imagedot{}edu \\
  \url{https://faculty.utrgv.edu/alexey.garber/}
\end{authorinfo}
\begin{authorinfo}[nm]
  Neil Ma\~nibo\\
  Bielefeld University\\
  Bielefeld, Germany, and\\
  Open University\\
  Milton Keynes, UK\\
  cmanibo\imageat{}math\imagedot{}uni-bielefeld\imagedot{}de\\
  neil\imagedot{}manibo\imageat{}open\imagedot{}ac\imagedot{}uk\\
  \url{https://www.math.uni-bielefeld.de/~cmanibo/}
\end{authorinfo}
\end{dajauthors}

\end{document}